\documentclass[11pt]{article}
\usepackage{odonnell}
\newif\ifdraft
\drafttrue

\ifdraft
\newcommand{\rnote}[1]{}
\newcommand{\ynote}[1]{}
\newcommand{\xnote}[1]{}
\else
\newcommand{\rnote}[1]{\footnote{\color{blue}Ryan: {#1}}}
\newcommand{\ynote}[1]{\footnote{\color{red}Yuval: {#1}}}
\newcommand{\xnote}[1]{\footnote{\color{green}Xinyu: {#1}}}
\fi

\newcommand{\hamlevel}{\kappa}

\newcommand{\colors}{{\ell}}
\newcommand{\slice}[1]{\calU_{#1}}

\newcommand{\transpos}[1]{\mathrm{Trans}(#1)}

\DeclarePairedDelimiterX\diverg[2]{(}{)}{#1 \mathrel{}\mathclose{}\delimsize\|\mathopen{}\mathrel{} #2}

\begin{document}

\title{FKN theorem for the multislice, with applications}
\author{Yuval Filmus\thanks{Technion Computer Science Department. \texttt{yuvalfi@cs.technion.ac.il}. Taub Fellow --- supported by the Taub Foundations. The research was funded by ISF grant 1337/16.}}

\maketitle

\begin{abstract}
The Friedgut--Kalai--Naor (FKN) theorem states that if $f$ is a Boolean function on the Boolean cube which is close to degree~1, then $f$ is close to a \emph{dictator}, a function depending on a single coordinate. The author has extended the theorem to the \emph{slice}, the subset of the Boolean cube consisting of all vectors with fixed Hamming weight. We extend the theorem further, to the \emph{multislice}, a multicoloured version of the slice.

As an application, we prove a stability version of the edge-isoperimetric inequality for settings of parameters in which the optimal set is a dictator.
\end{abstract}

\section{Introduction}

The classical Friedgut--Kalai--Naor (FKN) theorem~\cite{FKN02} is a basic structural result in Boolean Function Analysis. It is a stability version of the following trivial result: the only Boolean functions on the Boolean cube $\{0, 1\}^n$ which have degree~1 are dictators, that is, functions depending on a single coordinate. The FKN theorem can be stated in two equivalent ways:
\begin{enumerate}
\item If $f\colon \{0,1\}^n \to \{0,1\}$ is $\epsilon$-close to degree~1, that is, $\|f^{>1}\|^2 = \epsilon$, then $f$ is $O(\epsilon)$-close to a Boolean dictator, that is, $\Pr[f \neq g] = O(\epsilon)$ for some Boolean dictator $g\colon \{0,1\}^n \to \{0,1\}$.
\item If $f\colon \{0,1\}^n \to \R$ is a degree~1 function which is $\epsilon$-close to Boolean, that is, $\E[\dist(f,\{0,1\})^2] = \epsilon$, then $f$ is $O(\epsilon)$-close to a Boolean dictator, that is, $\E[(f-g)^2] = O(\epsilon)$ for some Boolean dictator $g\colon \{0,1\}^n \to \{0,1\}$.
\end{enumerate}
In fact, using hypercontractivity, the error bound can be improved from $O(\epsilon)$ to $\epsilon + O(\epsilon^2)$.\ynote{Or indeed, to $\epsilon + O(\epsilon^d)$ for any $d$.}

The FKN theorem has been extended to many other domains: to graph products~\cite{ADFS03}, to the biased Boolean cube~\cite{JOW12,Nay14}, to sums of functions on disjoint variables~\cite{Rub12}, and to non-product domains: the symmetric group~\cite{EFF15a,EFF15b} and the slice~\cite{Fil16b}.

In this paper we extend it to the multislice, a generalization of the slice recently considered by the authors~\cite{FOW18a}.

\smallskip

Given positive integers $\hamlevel_1,\ldots,\hamlevel_\colors$ summing to $n$, the \emph{multislice} $\slice{\hamlevel}$ consists of all vectors in $[\colors]^n$ in which the number of coordinates equal to~$i$ is $\hamlevel_i$. When $\colors = 2$, this is just the slice, and when $\colors = n$, we obtain the symmetric group. In this paper, we focus on the case in which $\colors$ is constant, and furthermore the multislice is unbiased: $\hamlevel_1,\ldots,\hamlevel_n \geq \rho n$ for some constant $\rho > 0$. The biased case, in which the weights $\hamlevel_1,\ldots,\hamlevel_\colors$ are allowed to become arbitrarily small, is more difficult, since in this case the approximating function need not be a dictator; see~\cite{Fil16b} for more details.

In order to formulate the FKN theorem for the multislice, we need to generalize the concept of degree~1 function. There are several different routes to this generalization, all yielding the same class of functions:
\begin{enumerate}
\item Representation theory of the symmetric group: The multislice can be viewed as a permutation module of $S_n$. The representation theory of $S_n$ decomposes the space of functions on $S_n$ to isotypical components indexed by partitions of~$n$, which are partially ordered according to majorization. In the case of the slice, the degree~$d$ functions are those supported on the isotypical components corresponding to partitions in which the first part contains at least $n-d$ boxes. We can use the same definition on the mutlislice.
\item Polynomial degree: We can view the input to a function on the multislice as consisting of Boolean variables $x_{ji} = 1_{u_j=i}$. A function on the Boolean cube or on the slice has degree~$d$ if it can be represented as a polynomial of degree~$d$ over these variables. This definition carries over to the multislice.
\item Junta degree: A function on the Boolean cube or on the slice has degree~$d$ if it is a linear combination of $d$-juntas, that is, functions depending on $d$ coordinates. The same definition works on the multislice.
\end{enumerate}

Armed with the concept of degree~1 function, we can state our main theorem.

\begin{theorem} \label{thm:main}
Fix an integer $\colors \geq 2$ and a parameter $\rho > 0$. There exists a constant $N = N(\colors,\rho)$ for which the following hold. Let $\hamlevel_1,\ldots,\hamlevel_\colors \geq \rho n$ be integer weights summing to $n \geq N$.

If $f\colon \slice{\hamlevel} \to \R$ is a degree~1 function which satisfies $\E[\dist(f,\{0,1\})^2] = \epsilon$, then there exists a Boolean function $g\colon \slice{\hamlevel} \to \{0,1\}$, depending on a single coordinate, such that $\E[(f-g)^2] \leq \epsilon + O_{\colors,\rho}(\epsilon^2)$.

If $F\colon \slice{\hamlevel} \to \{0,1\}$ satisfies $\|F^{>1}\|^2 = \epsilon$ then there exists a Boolean function $g\colon \slice{\hamlevel} \to \{0,1\}$, depending on a single coordinate, such that $\Pr[F \neq g] \leq 4\epsilon + O_{\colors,\rho}(\epsilon^2)$.
\end{theorem}

(The definition of $F^{>1}$ appears at the end of \Cref{sec:spectral}.)

\subsection{Application to edge isoperimetry}

Let $A$ be an arbitrary subset of the multislice $\slice{\hamlevel}$. The (edge) \emph{expansion} of $A$ is
\[
 \Phi(A) = \Pr_{\substack{\bu \sim A \\ \btau \sim \transpos{n}}}[\bu^\btau \not \in A],
\]
where $\bu$ is a random point chosen from $A$, $\btau = (j_1\; j_2)$ is a random transposition in $S_n$, and $\bu^\btau$ is obtained from $\bu$ by switching the values of $\bu_{j_1}$ and $\bu_{j_2}$. In words, the expansion of $A$ is the probability that if we choose a random point of $A$ and switch two of its coordinates at random, we reach a point not in $A$.

The edge-isoperimetry question is the following:

\begin{center}
\emph{Given $0 < \alpha < 1$, which sets of size $\alpha |\slice{\hamlevel}|$ minimize the expansion?}
\end{center}

When $\alpha n = \sum_{i \in S} \hamlevel_i$ for some $S \subseteq [\colors]$, it is natural to conjecture that the sets of the form $A = \{ u : u_j \in S \}$ minimize the expansion, and this is indeed the case. Using our FKN theorem, we are able to show a stability version of this result: if a set of size $\alpha n$ has almost minimal expansion, then it is close to a set with minimal expansion.

\subsection*{Preliminaries} We use $\E$ to denote expectation. The distance of an element $x$ to a set $S$ is $\dist(x,S) = \min_{y \in S} |x-y|$. For a set $S$, the notation $S \pm \epsilon$ stands for $\{ x : \dist(x,S) \leq \epsilon \}$. A function is \emph{Boolean} if it is $\{0,1\}$-valued. The L2 triangle inequality is the inequality $(a+b)^2 \leq 2(a^2+b^2)$.

Let $\hamlevel_1,\ldots,\hamlevel_\colors$ be positive integers summing to $n$. The \emph{multislice} $\slice{\hamlevel}$ consists of all vectors $u \in [\colors]^n$ in which the number of coordinates equal to~$i$ is $\hamlevel_i$, for all $i \in [\colors]$. The multislice is \emph{$\rho$-balanced} if $\hamlevel_1,\ldots,\hamlevel_\colors \geq \rho n$.

We endow the multislice with the uniform measure. If $f$ is a function on the multislice, then its L2 norm is $\|f\| = \sqrt{\E[f^2]}$. We say that two functions $f,g$ are $\epsilon$-close if $\|f-g\|^2 \leq \epsilon$.

We can think of a function on the multislice as being defined over the set of Boolean variables $(x_{ji})_{\substack{j \in [n] \\ i \in [\colors]}}$, which encode an element $u \in \slice{\hamlevel}$ in the following way: $x_{ji} = 1$ if $u_j = i$. Thus $\sum_{i=1}^\colors x_{ji} = 1$ for all $j \in [n]$, and $\sum_{j=1}^n x_{ji} = \hamlevel_i$ for all $i \in [\colors]$. (When $\colors = n$, the multislice is the symmetric group $S_n$, and the $x_{ji}$ are the entries of the permutation matrix representing the input permutation.) 

Since $x_{j\colors} = 1 - \sum_{i=1}^{\colors-1} x_{ji}$, we don't need to include $x_{1\colors},\ldots,x_{n\colors}$ explicitly as inputs. This is the usual convention in the case of the slice ($\colors = 2$), in which the input consists of just $n$ Boolean variables $x_1,\ldots,x_n$.

\section{Degree 1 functions} \label{sec:degree-1}

In this section we propose several different definitions of degree~1 functions, and show that they are all equivalent. While similar results hold for degree~$d$ functions for arbitrary~$d$, we concentrate here on the case $d = 1$.

Throughout the section, we fix a multislice $\slice{\hamlevel}$ on $n$ points and $\colors \geq 2$ colours. 

\subsection{Spectral definition} \label{sec:spectral}

A partition of $n$ is a non-increasing sequence of positive integers summing to $n$. We represent a partition as a finite sequence, or as an infinite sequence $(\lambda_i)_{i=1}^\infty$ where all but finitely many entries are zero. We can think of $\hamlevel$ as a partition of $n$ by sorting it accordingly.
We say that a partition $\lambda$ majorizes a partition $\mu$, in symbols $\lambda \succeq \mu$, if $\lambda_1 + \cdots + \lambda_i \geq \mu_1 + \cdots + \mu_i$ holds for all $i \geq 1$. 

The multislice $\slice{\hamlevel}$ can be viewed as a permutation module of the symmetric group. The representation theory of the symmetric group gives an orthogonal decomposition of the vector space of real-valued functions on the multislice:
\[
 \R^{\slice{\hamlevel}} = \bigoplus_{\lambda \succeq \hamlevel} V^\lambda,
\]
where $\lambda$ goes over all partitions of $n$ majorizing $\hamlevel$. Furthermore, it is known that $V^{(n)}$ consists of all constant functions, and $V^{(n-1,1)}$ is spanned by functions of the form $x_{ji_1} - x_{ji_2}$~\cite[Chapter 2]{Sagan}.

\begin{definition} \label{def:spectral-1}
A function on the multislice has \emph{spectral degree one} if it lies in $V^{(n)} \oplus V^{(n-1,1)}$.
\end{definition}

The orthogonal decomposition corresponds to the level decomposition of functions on the Boolean cube. In particular, we will use the following notations, for a function $f$ on the multislice:

\begin{enumerate}
\item $f^{=0}$ is the projection of $f$ to $V^{(n)}$.
\item $f^{=1}$ is the projection of $f$ to $V^{(n-1,n)}$.
\item $f^{\leq 1} = f^{=0} + f^{=1}$, and $f^{>1} = f - f^{\leq 1}$.
\end{enumerate}

Since $V^{(n)}$ consists of all constant functions, $f^{=0}$ is the constant function $\E[f]$.

\subsection{Polynomial definition} \label{sec:polynomial}

We can view the multislice as a function in the Boolean variables $x_{ji}$, where $j$ ranges over $[n]$ and $i$ ranges over $[\colors]$, given by $x_{ji} = 1_{u_j = i}$.

\begin{definition} \label{def:poly-1}
A function on the multislice has \emph{polynomial degree one} if it can be represented as a polynomial of degree at most~$1$ in the variables $x_{ji}$.
\end{definition}

Note that since $x_{j\colors} = 1 - \sum_{i=1}^{\colors-1} x_{ji}$, we can assume that the variables $x_{1\colors},\ldots,x_{n\colors}$ do not appear in the polynomial representation.

\begin{lemma} \label{lem:spectral-poly}
A function on the multislice has spectral degree one iff it has polynomial degree one.
\end{lemma}
\begin{proof}
If a function has spectral degree one then it is an affine combination of functions of the form $x_{ji_1} - x_{ji_2}$, and so it has polynomial degree one.

Conversely, suppose that $f$ has polynomial degree one, so that
\[
 f = c + \sum_{j=1}^n \sum_{i=1}^\colors c_{ji} x_{ji}.
\]
Let $c_j = \sum_{i=1}^\colors c_{ji}/\colors$ for all $j \in [n]$.
Since $\sum_{i=1}^\colors x_{ji} = 1$, we have
\[
 f = c + \sum_{j=1}^n \colors c_j + \sum_{j=1}^n \sum_{i=1}^\colors (c_{ji} - c_j) x_{ji}.
\]
By construction, $\sum_{i=1}^\colors (c_{ji} - c_j) = 0$, and so $c_{j\colors} - c_j = -\sum_{i=1}^{\colors-1} (c_{ji} - c_j)$. Therefore
\[
 f = c + \sum_{j=1}^n \colors c_j + \sum_{j=1}^n \sum_{i=1}^{\colors-1} (c_{ji} - c_j) (x_{ji} - x_{j\colors}). 
\]
This shows that $f$ has spectral degree one.
\end{proof}

\subsection{Junta definition} \label{sec:junta}

A \emph{dictator} is a function depending on a single coordinate. This also includes constant functions.

\begin{definition} \label{def:junta-1}
A function on the multislice has \emph{junta degree one} if it can be represented as a linear combination of dictators.
\end{definition}

\begin{lemma} \label{lem:poly-junta}
A function on the multislice has polynomial degree one iff it has junta degree one.
\end{lemma}
\begin{proof}
The functions $1,x_{ji}$ are dictators, and so if a function has polynomial degree one then it has junta degree one. Conversely, if $f$ depends only on the $j$th coordinate then $f = \sum_{i=1}^\colors c_i x_{ji}$ for some constants $c_1,\dots,c_\colors$, and so $f$ has polynomial degree one. Therefore a function having junta degree one also has polynomial degree one.
\end{proof}

In view of \Cref{lem:spectral-poly} and \Cref{lem:poly-junta}, we define a function on the multislice to have degree one if it satisfies any of the definitions given above.

\subsection{Normal form} \label{sec:normal-form}

We close this section by describing a normal form for degree one functions.

\begin{lemma} \label{lem:normal-form-1}
Every degree one function on the multislice has a unique representation of the form
\[
 f = c + \sum_{j=1}^n \sum_{i=1}^\colors c_{ji} x_{ji},
\]
where $\sum_{i=1}^\colors c_{ji} = 0$ for all $j \in [n]$ and $\sum_{j=1}^n c_{ji} = 0$ for all $i \in [\colors]$.
\end{lemma}
\begin{proof}
We start by showing that if $f$ has degree one then it has a representation as required by the lemma. By linearity, it suffices to show this for the function $x_{11}$ (the lemma clearly holds for constant functions). Since $\sum_{i=1}^\colors x_{1i} = 1$, we have
\[
 x_{11} = \frac{1}{\colors} \sum_{i=1}^\colors (x_{11} - x_{1i}) + \frac{1}{\colors}.
\]
Similarly, since $\sum_{j=1}^n x_{ji} = \hamlevel_i$, we have
\[
 x_{1i} = \frac{1}{n} \sum_{j=1}^n (x_{1i} - x_{ji}) + \frac{\hamlevel_i}{n}.
\]
Combining both expressions together, we obtain
\begin{align*}
 x_{11} &= \frac{1}{\colors n} \sum_{i=1}^\colors \sum_{j=1}^n (x_{11} - x_{j1} - x_{1i} + x_{ji}) + \frac{1}{\colors n} \sum_{i=1}^\colors (\hamlevel_1 - \hamlevel_i) + \frac{1}{\colors} \\ &=
 \frac{1}{\colors n} \sum_{i=1}^\colors \sum_{j=1}^n (x_{11} - x_{j1} - x_{1i} + x_{ji}) + \frac{\hamlevel_1}{n}.
\end{align*}
It is not hard to check that $x_{11} - x_{j1} - x_{1i} + x_{ji}$ satisfies the requisite properties for all $j \in [n]$ and $i \in [\colors]$, hence so does the expression given for $x_{11}$.

Next, we show that the representation is unique. It suffices to show that the only representation of the zero function is the zero polynomial. In other words, we have to show that if
\[
 0 = c + \sum_{j=1}^n \sum_{i=1}^\colors c_{ji} x_{ji},
\]
where the $c_{ji}$ satisfy all the constraints in the lemma, then $c = 0$ and $c_{ji} = 0$ for all $j \in [n]$ and $i \in [\colors]$.

Choose any two indices $j_1 \neq j_2$ and any two colors $i_1 \neq i_2$. Consider an arbitrary point $u$ in the multislice satisfying $u_{j_1} = i_1$ and $u_{j_2} = i_2$, and the point obtained by switching $i_1$ and $i_2$. Subtracting the corresponding right-hand sides, we deduce
\[
 0 = c_{j_1i_1} - c_{j_1i_2} - c_{j_2i_1} + c_{j_2i_2}.
\]
This identity also holds when $j_1 = j_2$.
Averaging over all values of $j_2$ and using $\sum_{j=1}^n c_{ji_1} = \sum_{j=1}^n c_{ji_2} = 0$, we deduce that $c_{j_1i_1} = c_{j_1i_2}$ for all $i_1 \neq i_2$. Since $\sum_{i=1}^\colors c_{j_1i} = 0$, this implies that $c_{j_1i} = 0$ for all $i \in [\colors]$ and all $j_1 \in [n]$. It follows that also $c = 0$, and so the only representation of zero is the zero polynomial, completing the proof of uniqueness.
\end{proof}

\section{FKN theorem} \label{sec:fkn}

In this section we prove \Cref{thm:main}, by induction on the number of colours. The actual statement that we will prove by induction is the following.

\begin{theorem} \label{thm:main-proof}
Fix an integer $\colors \geq 2$ and a parameter $\rho > 0$. There exists a constant $N = N(\colors,\rho)$ such that for every $\rho$-balanced multislice on $n \geq N$ points and $\colors$ colours, the following holds.

If $f\colon \slice{\hamlevel} \to \R$ is a degree~1 function which satisfies $\E[\dist(f,\{0,1\})^2] = \epsilon$, then there exists a Boolean dictator $g$ such that $\E[(f-g)^2] = O_{\colors,\rho}(\epsilon)$.
\end{theorem}


\Cref{thm:main} follows from this formulation using the following argument.
\begin{proof}[Proof of \Cref{thm:main}]
We start with the first part of the theorem. Let $f\colon \slice{\hamlevel} \to \R$ be a degree one function which satisfies $\E[\dist(f,\{0,1\})^2] = \epsilon$. \Cref{thm:main-proof} shows that there exists a Boolean dictator $g$ such that $\E[(f-g)^2] = O(\epsilon)$. Let $h = f - g$. Since $g$ is Boolean, $\E[\dist(h,\{0,\pm1\})^2] \leq \epsilon$. When $|h| \leq 1/2$, we have $\dist(h,\{0,\pm1\})^2 = h^2$, and so
\[
 \epsilon \geq \E[\dist(h,\{0,\pm1\})^2] \geq \E[h^2 1_{|h| \leq 1/2}] = \E[h^2] - \E[h^2 1_{|h|>1/2}].
\]
When $|h| > 1/2$, we have $h^4 > h^2/4$, and so
\[
 \E[h^2 1_{|h| > 1/2}] < 4 \E[h^4].
\]
In~\cite{FOW18a} it is shown that a $\rho$-biased multislice is hypercontractive for any constant $\rho$ and constant number of colours, and so $\E[h^4] = O_{\colors,\rho}(\E[h^2]^2) = O_{\colors,\rho}(\epsilon^2)$, since $h$ has degree one. This shows that $\E[h^2] < \epsilon + 4\E[h^4] = \epsilon + O_{\colors,\rho}(\epsilon^2)$, completing the proof of the first part of the theorem.

We continue with the second part of the theorem, which is very similar. Let $F\colon \slice{\hamlevel} \to \R$ be a Boolean function which satisfies $\|F^{>1}\|^2 = \epsilon$, and let $f = F^{\leq 1}$. Since $\E[\dist(f,\{0,1\})^2] \leq \E[(f-F)^2] = \epsilon$, the first part gives a Boolean dictator $g$ satisfying $\E[(f-g)^2] = \epsilon + O_{\colors,\rho}(\epsilon^2)$. The L2 triangle inequality implies that $\E[(F-g)^2] \leq 4\epsilon + O_{\colors,\rho}(\epsilon^2)$. Since both $F$ and $g$ are Boolean, $\Pr[F \neq g] = \E[(F-g)^2]$, completing the proof.
\end{proof}

For brevity, in the rest of the section we use $O(\cdot)$ for $O_{\colors,\rho}(\cdot)$.

\subsection{Base case}

The base case of our inductive proof is when $\colors = 2$, and it follows from the main result of~\cite{Fil16b}, whose statement reads as follows.

\begin{theorem} \label{thm:fkn-slice}
Suppose that $f\colon \slice{k,n-k} \to \{0,1\}$ satisfies $\|f^{>1}\|^2 = \epsilon$, where $2 \leq k \leq n/2$. Then either $f$ or $1-f$ is $O(\epsilon)$-close to a function of the form $\max_{i \in S} x_{i1}$, where $S \subseteq [n]$ has cardinality at most $\max(1,O(\sqrt{\epsilon}/(k/n)))$.
\end{theorem}

From this theorem, we deduce the base case of \Cref{thm:main-proof}.

\begin{proof}[Proof of \Cref{thm:main-proof} in the case $\colors=2$]
Let $f\colon \slice{k,n-k} \to \R$ be a degree one function satisfying $\E[\dist(f,\{0,1\})^2] = \epsilon$, and assume without loss of generality that $k \leq n/2$. Let $F$ be the function obtained by rounding $f$ to $\{0,1\}$. By definition, $\E[(F-f)^2] = \epsilon$, and so $\|F^{>1}\|^2 \leq \epsilon$ (this is since $F^{\leq 1}$ is the degree one function which is closest to $F$).

By choosing $N$ appropriately, we can ensure that $k \geq 2$, and so \Cref{thm:fkn-slice} applies, showing that either $f$ or $1-f$ is $O(\epsilon)$-close to a function depending on at most $\max(1,m)$ coordinates, where $m = O(\sqrt{\epsilon}/(k/n)) = O_\rho(\sqrt{\epsilon})$.

We now consider two cases. The first case is when $m \leq 1$. In this case, $F$ is $O(\epsilon)$-close to a dictator. Since $\E[(F-f)^2] = \epsilon$, it follows that $f$ is also $O(\epsilon)$-close to the same dictator.

When $m > 1$, we can lower bound $\epsilon \geq e_\rho$ for some constant $e_\rho > 0$ depending on $\rho$. The L2 triangle inequality implies that
\[
 \E[f^2 1_{F=1}] \leq 2\E[(f-1)^2 1_{F=1}] + 2.
\]
Therefore
\[
 \epsilon = \E[\dist(f,\{0,1\})^2] =
 \E[f^2 1_{F=0}] + \E[(f-1)^2 1_{F=1}] \geq
 \E[f^2 1_{F=0}] + \frac{1}{2} \E[f^2 1_{F=1}] - 1 \geq
 \frac{1}{2} \E[f^2] - 1.
\]
In other words, $\E[f^2] \leq 2(1+\epsilon)$. This implies that
\[
 \|f-0\|^2 \leq \frac{2(1+\epsilon)}{\epsilon} \epsilon \leq \frac{2(1+e_\rho)}{e_\rho} \epsilon,
\]
completing the proof in this case.
\end{proof}

\subsection{Inductive step}

We now assume that \Cref{thm:main-proof} holds for a certain value of $\colors \geq 2$, and will prove it for $\colors + 1$.

We start with a simple comment: \Cref{thm:main-proof} is trivial for large $\epsilon$, using the same argument used to derive the second part of \Cref{thm:main}. Indeed, suppose that $\epsilon \geq \epsilon_0$. Then
\[
 \E[f^2 1_{|f|>1/2}] \leq 2\E[(f-1)^2 1_{|f|>1/2}] + 2,
\]
and so
\begin{multline*}
 \epsilon = \E[\dist(f,\{0,1\})^2] = \E[f^2 1_{|f| \leq 1/2}] + \E[(f-1)^2 1_{|f|>1/2}] \geq \\
 \E[f^2 1_{|f| \leq 1/2}] + \frac{1}{2} \E[f^2 1_{|f| > 1/2}] - 1 \geq \frac{1}{2} \E[f^2] - 1,
\end{multline*}
implying that $\E[f^2] \leq 2(1+\epsilon)$. Therefore
\[
 \E[(f-0)^2] \leq \frac{2(1+\epsilon)}{\epsilon} \epsilon \leq \frac{2(1+\epsilon_0)}{\epsilon_0} \epsilon.
\]
Since $0$ is a dictator, we see that when $\epsilon \geq \epsilon_0$, \Cref{thm:main-proof} trivially holds. Therefore, from now on we may assume that $\epsilon$ is small enough (as a function of $\colors$ and $\rho$).

Next, we need a criterion that guarantees that the approximating function in \Cref{thm:main-proof} is constant. We will use the concept of \emph{influence}: given two coordinates $j_1,j_2 \in [n]$ and a function $f\colon \slice{\hamlevel} \to \R$,
\[
 \Inf_{j_1j_2}[f] = \E_{\bu \sim \slice{\hamlevel}}\bigl[\bigl(f(\bu) - f(\bu^{(j_1\;j_2)})\bigr)^2\bigr].
\]

\begin{lemma} \label{lem:main-constant}
Let $\slice{\hamlevel}$ be a $\rho$-balanced multislice with $\colors$ colours. There exists a constant $\eta = \eta(\rho,\colors)$ such that the following holds for all $\epsilon \leq \eta$. If $f\colon \slice{\hamlevel} \to \R$ is a degree $1$ function which satisfies $\E[\dist(f,\{0,1\})^2] = \epsilon$ and $\Inf_{j_1j_2}[f] \leq \eta$ for all $j_1,j_2 \in [n]$ then there exists a constant $C \in \{0,1\}$ such that $\E[(f-C)^2] = O(\epsilon)$ and $|\E[f] - C| = O(\sqrt{\epsilon})$.
\end{lemma}
\begin{proof}
 \Cref{thm:main-proof} shows the existence of a Boolean dictator $g$ satisfying $\E[(f-g)^2] = O(\epsilon)$. The L2 triangle inequality shows that $\Inf_{j_1j_2}[g] = O(\Inf_{j_1j_2}[f] + \epsilon) = O(\eta)$. Suppose, for the sake of contradiction, that $g$ isn't constant. Then there exists a coordinate $j_1$ and colours $i_1,i_2$ such that $g(u) = 0$ if $u_{j_1} = i_1$ and $g(u) = 1$ if $u_{j_1} = i_2$. Let $j_2$ be any other coordinate. A random $\bu$ chosen from the multislice satisfies $\bu_{j_1} = i_1$ and $\bu_{j_2} = i_2$ with probability $\Omega(\rho^2)$. When that happens, $(g(\bu) - g(\bu^{(j_1\;j_2)}))^2 = 1$. Therefore $\Inf_{j_1j_2}[g] = \Omega(\rho^2)$. By choosing $\eta$ small enough, we reach a contradiction. We conclude that $g = C$ for some constant $C \in \{0,1\}$.
 
 The L1--L2 norm inequality implies that $\E[|f-C|]^2 \leq \E[(f-C)^2] = O(\epsilon)$, and so $|\E[f] - C| \leq \E[|f-C|] = O(\sqrt{\epsilon})$, completing the proof.
\end{proof}

\subsubsection{Isolating the dictatorial coordinate}

The first step in the argument is to identify the dictatorial coordinate, if any. We do this by looking at the degree one expansion of $f$:
\[
 f = c + \sum_{j=1}^n \sum_{i=1}^{\colors} c_{ji} x_{ji}.
\]
Note that although there are $\colors+1$ colours, using the identity $x_{j(\colors+1)} = 1 - \sum_{i=1}^\colors x_{ji}$ we can eliminate all variables involving the last colour.

Let $j_1 \neq j_2$ be two arbitrary coordinates, and let $i \neq \colors+1$ be an arbitrary color. Suppose that $u$ is an element of the multislice satisfying $u_{j_1} = i$ and $u_{j_2} = \colors+1$. A short calculation shows that
\[
 f(u) - f(u^{(j_1\;j_2)}) = c_{j_1i} - c_{j_2i}.
\]
When choosing $\bu$ at random from the multislice, the event $\bu_{j_1} = i$ and $\bu_{j_2} = \colors+1$ occurs with probability $\Omega(\rho^2)$. Therefore the L2 triangle inequality implies that
\[
 4\epsilon \geq \E[\dist(f(u) - f(u^{(j_1\;j_2)}), \{0,\pm 1\})^2] = \Omega(\rho^2 (c_{j_1i} - c_{j_2i})^2),
\]
implying that $\dist(c_{j_1i} - c_{j_2i}, \{0,\pm1\}) = O(\sqrt{\epsilon})$.
Choosing $c_i := \min_j c_{ji}$, we deduce that $c_{ji} \in \{c_i,c_i+1\} \pm O(\sqrt{\epsilon})$ for all $j \in [n]$.

We associate with each coordinate $j \in [n]$ a vector $\gamma_j \in \{0,1\}^\colors$ such that $|c_{ji} - c_i - \gamma_{ji}| = O(\sqrt{\epsilon})$. Assuming $n > 2^\colors$, there exists a vector $v \in \{0,1\}^\colors$ which is realized by at least two coordinates $j_1,j_2$. Our goal now is to show that $v$ is realized by all but at most one coordinate. To this end, let us assume that $\gamma_{J_1},\gamma_{J_2} \neq v$ for some coordinates $J_1 \neq J_2$. Let $i_1,i_2 \neq \colors+1$ be colours such that $\gamma_{J_1i_1} \neq v_{i_1}$ and $\gamma_{J_2i_2} \neq v_{i_2}$.

If an element $u$ of the multislice satisfies $\{u_{J_1},u_{j_1}\} = \{i_1,\colors+1\}$ then $f(u) - f(u^{(J_1\;j_1)}) = \pm 1 \pm O(\sqrt{\epsilon})$, and similarly for $J_2,j_2$. Hence we can find a constraint on $u_{J_1},u_{j_1},u_{J_2},u_{j_2}$ which implies $f(u) - f(u^{(J_1\;j_1)(J_2\;j_2)}) = 2 \pm O(\sqrt{\epsilon})$. For small enough $\epsilon$, this guarantees that $\dist(f(u) - f(u^{(J_1\;j_1)(J_2\;j_2)}), \{0,\pm1\})^2 \geq 1/2$. A random $\bu \sim \slice{\hamlevel}$ satisfies the constraint with probability $\Omega(\rho^4)$, and so
\[
 4\epsilon \geq \E[\dist(f(u) - f(u^{(J_1\;j_1)(J_2\;j_2)}), \{0,\pm 1\})^2] = \Omega(\rho^4),
\]
which is impossible if $\epsilon$ is small enough.

We conclude that $\gamma_j = v$ for all but at most a single coordinate. Without loss of generality, let the exceptional coordinate (if any) be the last coordinate.

\subsubsection{Constant pieces}

Our strategy now is to consider restrictions of $f$ obtained by fixing the value of the last coordinate. For large enough $n$, fixing the last coordinate to colour $I \in [\colors+1]$ will result in a function $f_I$ on a $(\rho/2)$-balanced multislice $\hamlevel^{(I)}$ on $n-1$ points and $\colors+1$ colours. Let $\epsilon_I = \E[\dist(f_I,\{0,1\})^2]$. We will show that \Cref{thm:main-proof} holds for each $f_I$, and later on put all pieces together. Just as above, we can assume that $\epsilon_I$ is small enough.

Let us start by noting that
\[
 f_I = c^{(I)} + \sum_{j=1}^{n-1} \sum_{i=1}^\colors c_{ji} x_{ji},
\]
where the coefficients $c_{ji}$ are the same as before. Suppose now that $S \subseteq [n-1]$ is a set of $\hamlevel^{(I)}_{\colors+1}$ coordinates, and let $S' = S \cup \{n\}$. Let $f_{I,S}$ be the function obtained by setting all coordinates in $S$ to the value $\colors+1$:
\[
 f_{I,S} = c^{(I)} + \sum_{j \notin S'} \sum_{i=1}^\colors c_{ji} x_{ji}.
\]
This is a function on a $(\rho/2)$-balanced multislice $\hamlevel^{(I,S)}$ on $\colors$ colours, so we can apply \Cref{thm:main-proof} or its corollary, \Cref{lem:main-constant}. In preparation for such an application, let us define $\epsilon_{I,S} = \E[\dist(f_{I,S},\{0,1\})^2]$.

By construction, for each $i \in [\colors]$ there exists a value $d_i \in \{c_i,c_i+1\}$ such that $|c_{ji} - d_i| = O(\sqrt{\epsilon})$ for all $j \in [n-1]$. This allows us to upper-bound $\Inf_{j_1j_2}[f_{i,S}]$ for all coordinates $j_1,j_2$. Indeed, if $u_{j_1} = i_1$ and $u_{j_2} = i_2$ then
\[
 |f_{I,S}(u) - f_{I,S}(u^{(j_1\;j_2)})| = 
 |c_{j_1i_1} + c_{j_2i_2} - c_{j_1i_2} - c_{j_2i_1}| = O(\sqrt{\epsilon}).
\]
This shows that $\Inf_{j_1j_2}[f_I] = O(\epsilon)$. For small enough $\epsilon$, this allows us to apply \Cref{lem:main-constant} in order to conclude that there is a constant $C_{I,S} \in \{0,1\}$ such that $\E[(f_{I,S} - C_{I,S})^2] = O(\epsilon_{I,S})$ and $|\E[f_{I,S}] - C_{I,S}| = O(\sqrt{\epsilon_{I,S}})$.

We apply the foregoing to a random choice $\bS$. The next step is to show that $C_{I,\bS}$ is concentrated. To this end, we calculate
\[
 \E[f_{I,\bS}] = c^{(I)} + \sum_{j \notin \bS\boldmath{'}} \sum_{i=1}^\colors c_{ji} \frac{\hamlevel^{(I)}_i}{m},
\]
where $m = \sum_{i=1}^\colors \hamlevel^{(I)}_i$. We can view $\E[f_{I,\bS}]$ as a function on the multislice $\slice{m,\hamlevel^{(I)}_{\colors+1}}$. Denoting it by $\mu$ and using a different parametrization of the slice, we have
\[
 \mu = c^{(I)} + \sum_{j=1}^{n-1} x_j \sum_{i=1}^\colors c_{ji} \frac{\hamlevel^{(I)}_i}{m}.
\]
This is a degree one function, and it satisfies
\[
 \E[\dist(\mu,\{0,1\})^2] \leq \E[(\mu - C_{I,\bS})^2] = O(\E[\epsilon_{I,\bS}]) = O(\epsilon_I).
\]
Furthermore, for each $j_1 \neq j_2$ we have
\[
 \Inf_{j_1j_2}[\mu] \leq \left( \sum_{i=1}^\colors (c_{j_1i} - c_{j_2i}) \frac{\hamlevel^{(I)}_i}{m} \right)^2 = O(\epsilon).
\]
For small enough $\epsilon$, we can thus apply \Cref{lem:main-constant} (for two colours) to deduce that $\E[(\mu - C_I)^2] = O(\epsilon_I)$ for some constant $C_I \in \{0,1\}$.

Without loss of generality, let us suppose that $C_I = 0$. Then $\E[\mu^2] = O(\epsilon_I)$, and so $\Pr[\mu \geq 1/2] = O(\epsilon_I)$. This shows that $C_{I,\bS} = 1$ with probability $O(\epsilon_I)$. Therefore
\begin{multline*}
 \E[f_I^2] = \E_{\bS} \bigl[\E[f_{I,\bS}^2]\bigr] =
 \E_{\bS}\bigl[\E[(f_{I,\bS} - C_{I,\bS})^2 1_{C_{I,\bS}=0}]\bigr] +
 \E_{\bS}\bigl[\E[(f_{I,\bS} - C_{I,\bS} + 1)^2 1_{C_{I,\bS}=1}]\bigr] \leq \\
 2\E_{\bS}\bigl[\E[(f_{I,\bS} - C_{I,\bS})^2]\bigr] + 2\Pr[C_{I,\bS} = 1] =
 O\bigl(\E_{\bS}[\epsilon_{I,\bS}]\bigr) + O(\epsilon_I) = O(\epsilon_I).
\end{multline*}

Taking also the case $C_I = 1$ into account, we deduce
\[
 \E[(f_I - C_I)^2] = O(\epsilon_I), \quad C_I \in \{0,1\}.
\]

\subsubsection{Completing the proof}

We can now complete the proof of \Cref{thm:main-proof}. Let $g(u) = C_{u_n}$, a Boolean dictator. Let $\bI$ be the marginal distribution of $\bu_n$ when $\bu \sim \slice{\hamlevel}$. Then
\[
 \E[(f-g)^2] = \E_{\bI}\bigl[\E[(f_{\bI} - C_{\bI})^2]\bigr] = O\bigl(\E_{\bI}[\epsilon_{\bI}]\bigr) = O(\epsilon).
\]
This completes the proof.

\section{Edge isoperimetry} \label{sec:isoperimetry}

Consider a multislice $\slice\hamlevel$ on at least $4$ points. Define the \emph{volume} of a subset $A$ of the multislice $\slice{\hamlevel}$ to be $\vol(A) = |A|/|\slice{\hamlevel}|$.
The goal of this section is to prove the following isoperimetric inequality: if $\vol(A) = \alpha$ then
\[
 \Phi(A) \geq \frac{2(1-\alpha)}{n-1}.
\]
We will also identify when this inequality is tight, and prove stability in these cases.

\subsection{Spectral formula} \label{sec:spectral-formula}

For a partition $\lambda \succeq \hamlevel$, let us denote by $f^{=\lambda}$ the orthogonal projection of $f$ to $V^{\lambda}$ (see \Cref{sec:spectral} for the appropriate definitions).
Frobenius~\cite{Fro00} proved the following formula:
\begin{equation} \label{eq:frobenius}
 \E_{\btau \sim \transpos{n}}[f^\btau] = \sum_{\lambda \succeq \hamlevel} c_\lambda f^{=\lambda}, \text{ where } c_\lambda = \frac{1}{n(n-1)} \sum_{i=1}^\colors [\lambda_i^2 - (2i-1)\lambda_i].
\end{equation}

A classical fact is that $c_\lambda > c_\mu$ if $\lambda \succ \mu$; see for example~\cite[Lem.~10]{DS81}. This allows us to identify the minimal values of $1 - c_\lambda$.

\begin{lemma} \label{lem:extreme-eigenvalues}
We have $1-c_{(n)} = 0$, $1-c_{(n-1,1)} = \frac{2}{n-1}$, and $1-c_\lambda \geq \frac{4}{n}$ for all $\lambda \neq (n),(n-1,1)$.
\end{lemma}
\begin{proof}
The largest three partitions in majorization order are $(n),(n-1,1),(n-2,2)$. Calculation shows that $c_{(n-2,2)} = 4/n$, and so the lemma follows from the observation that $c_\lambda > c_\mu$ if $\lambda \succ \mu$.
\end{proof}

The important formula of Frobenius allows us to deduce one for $\Phi(A)$.

\begin{lemma} \label{lem:expansion-formula}
For any $A \subseteq \slice{\hamlevel}$,
\[
 \Phi(A) = \frac{1}{\vol(A)} \sum_{\lambda \succeq \hamlevel} (1-c_\lambda) \|1_A^{=\lambda}\|^2.
\]
\end{lemma}
\begin{proof}
For a given element $\bu \in A$ and a given transposition $\btau \in \transpos{n}$, the element $\bu^\btau$ lies in $A$ if $\E[1_{\bu^\btau} 1_A] = 1/|\slice\hamlevel|$, and otherwise $\E[1_{\bu^\btau} 1_A] = 0$. Hence
\[
\Pr_{\bu \sim A}[\bu^\btau \in A] =
\frac{|\slice\hamlevel|}{|A|} \E[1_A^\btau 1_A] = \frac{1}{\vol(A)} \langle 1_A^\btau, 1_A \rangle.
\]
Averaging over $\btau$, we get
\[
\Pr_{\substack{\bu \sim A \\ \btau \sim \transpos{n}}}[\bu^\btau \in A] =
\frac{1}{\vol(A)} \left\langle \E_{\btau \sim \transpos{n}}[1_A^\btau], 1_A \right\rangle.
\]
Applying \Cref{eq:frobenius} and the orthogonality of the isotypical decomposition, we obtain
\[
\Pr_{\substack{\bu \sim A \\ \btau \sim \transpos{n}}}[\bu^\btau \in A] =
\frac{1}{\vol(A)} \sum_{\lambda \succeq \hamlevel} \langle c_\lambda 1_A^{=\lambda}, 1_A^{=\lambda} \rangle = \frac{1}{\vol(A)} \sum_{\lambda \succeq \hamlevel} c_\lambda \|1_A^{=\lambda}\|^2.
\]
The lemma now follows from the identity $\|1_A\|^2 = \sum_\lambda \|1_A^{=\lambda}\|^2$.
\end{proof}

\subsection{Main argument} \label{sec:hoffman}

\Cref{lem:expansion-formula} implies an isoperimetric inequality, along the lines of Hoffman's bound.

\begin{lemma} \label{lem:isoperimetric}
If $A \subseteq \slice{\hamlevel}$ and $\vol(A) = \alpha$ then
\[
 \Phi(A) \geq \frac{2(1 - \alpha)}{n-1}.
\]
Furthermore, if $\hamlevel_1,\ldots,\hamlevel_\colors \geq 2$ and the bound is tight then $A$ is a dictator (membership in $A$ depends on the colour of a single coordinate).

Suppose now that the number of colours is bounded, and that the multislice is $\rho$-balanced for some constant $\rho$.
If $\Phi(A) \leq (1+\epsilon)\frac{2(1-\alpha)}{n-1}$ (where $\epsilon>0$) then there exists a Boolean dictator $B$ such that
\[
 \frac{|A \triangle B|}{|\slice{\hamlevel}|} = O(\alpha(1-\alpha)\epsilon).
\]
\end{lemma}
\begin{proof}
 Since $1_A^{=(n)} = \E[1_A] \bone = \vol(A) \bone$, it follows that $\|1_A^{=(n)}\|^2 = \vol(A)^2$. Similarly, $\sum_{\lambda \succeq \hamlevel} \|1_A^{=\lambda}\|^2 = \|1_A\|^2 = \E[1_A] = \vol(A)$. Hence combining \Cref{lem:expansion-formula} and \Cref{lem:extreme-eigenvalues}, we have
\[
 \Phi(A) \geq \frac{1}{\vol(A)} \cdot \frac{2}{n-1} (\vol(A) - \vol(A)^2) = \frac{2(1 - \vol(A))}{n-1}.
\]
This proves the upper bound. If the upper bound is tight, then $1_A$ is supported on $(n),(n-1,1)$, and so $1_A$ has degree one. This implies~\cite{FI18a} that $A$ is a dictator.

Suppose now that the number of colours is bounded, that the multislice is $\rho$-balanced, and that $\Phi(A) \leq (1+\epsilon)\frac{2(1-\alpha)}{n-1}$. Let $\delta = \|1_A\|^2 - \|1_A^{=(n)}\|^2 - \|1_A^{=(n-1,1)}\|^2$. Then
\begin{align*}
 \Phi(A) &\geq \frac{1}{\vol(A)} \cdot
 \left[
  \frac{2}{n-1} (\vol(A) - \vol(A)^2) +
  \frac{2(n-2)}{n(n-1)} \delta
 \right] \\ &=
 \frac{2(1-\alpha)}{n-1} + \frac{2(n-2)}{n(n-1)} \frac{\delta}{\alpha}.
\end{align*}
The assumption on $\Phi(A)$ thus implies an upper bound on $\delta$:
\[
 \delta \leq \frac{n\alpha (1-\alpha)}{(n-2)} \epsilon = O(\alpha(1-\alpha)\epsilon).
\]
\Cref{thm:main} shows that if $n$ is larger than some constant depending on $\colors$ and $\rho$ then $1_A$ is $O(\delta)$-close to a Boolean dictator $1_B$, completing the proof. When $n$ is small, compactness shows that $1_A$ is trivially $O(\delta)$-close to $\emptyset$, since there are only finitely many possible $A,\alpha,\epsilon$.
\end{proof}

\begin{corollary} \label{cor:isoperimetric}
Suppose that $\alpha \in (0,1)$ satisfies $\alpha n = \sum_{i \in S} \hamlevel_i$ for some $S \subseteq [\colors]$. Then the bound in \Cref{lem:isoperimetric} is tight for the families
\[
 A_{j,S} = \{ u : u_j \in S \}, \quad j \in [n].
\]
Conversely, if the bound in \Cref{lem:isoperimetric} is tight for a family $A$ then there exists a set $S \subseteq [\colors]$ satisfying $\alpha n = \sum_{i \in S} \hamlevel_i$ and a coordinate $j \in [n]$ such that $A = A_{j,S}$.
\end{corollary}
\begin{proof}
 The expansion of $A_{j,S}$ is the probability that a random transposition is of the form $(j\;k)$, where $k$ is one of the $(1-\alpha)n$ coordinates whose colour is not in $S$. Therefore
\[
 \Phi(A_{j,S}) = \frac{(1-\alpha) n}{\binom{n}{2}} = \frac{2(1-\alpha)}{n-1}.
\]
This shows that the bound in \Cref{lem:isoperimetric} is tight for $A_{j,S}$.

Conversely, if the bound in \Cref{lem:isoperimetric} is tight for a family $A$ then $A$ is a dictator, and so of the form $A_{j,S}$. Since $\vol(A_{j,S}) = \sum_{i \in S} \hamlevel_i/n$, we see that $\sum_{i \in S} \hamlevel_i = \alpha n$.
\end{proof}

\bibliographystyle{alpha}
\bibliography{odonnell-bib}

\end{document}